\documentclass[11pt]{amsart}
\usepackage[utf8]{inputenc}
\usepackage{graphicx,xcolor,verbatim}
\usepackage{amscd,amsmath,amsfonts,amssymb}
\usepackage{mathtools}
\usepackage[foot]{amsaddr}
\usepackage[pdfborder=000,pdftex=true]{hyperref}

\usepackage{cite}

\newtheorem{theorem}{Theorem}

\newtheorem{remark}[theorem]{Remark}

\numberwithin{equation}{section}

\newcommand{\R}{\mathbb R}
\newcommand{\N}{\mathbb N}
\newcommand{\essinf}{\mathop{\rm essinf}}

\newcommand{\SD}{\Sigma_{\mathcal D}}

\newcommand{\SN}{\Sigma_{\mathcal N}}
\newcommand{\Hs}{H_{\Sigma_{\mathcal{D}}}^s(\Omega)}

\begin{document}
	
\title[Asymptotically linear fractional problems]{Asymptotically linear fractional problems\\ with mixed boundary conditions}

\author{Giovanni Molica Bisci}
\address[G. Molica Bisci]{Department of Human Sciences and Promotion of Quality of Life, San Raffaele University, via di Val Cannuta 247, I-00166 Roma, Italy}
\email{\tt giovanni.molicabisci@uniroma5.it}

\author{Alejandro Ortega}
\address[A. Ortega]{Dpto. de Matem\'aticas Fundamentales, Facultad de Ciencias, UNED, 28040 Madrid, Spain}
\email{\tt alejandro.ortega@mat.uned.es}

\author{Luca Vilasi$^\dagger$}
\address[L. Vilasi]{Department of Mathematical and Computer Sciences, Physical Sciences and Earth Sciences\\
University of Messina\\
Viale F. Stagno d’Alcontres, 31 - 98166 Messina, Italy}
\email{\tt lvilasi@unime.it}

\keywords{Fractional Laplacian, Variational Methods, Pseudo-index theory, Mixed Boundary Data, Asymptotically
linear problem.\\
\phantom{aa} 2020 AMS Subject Classification: Primary: 49J35, 35A15, 58E05; Secondary: 35J61, 35S15.\\
\phantom{aa} $^\dagger$Corresponding author: L. Vilasi.
}

\maketitle

\begin{abstract}
We derive the existence of solutions for an asymptotically linear equation driven by the spectral fractional Laplacian operator with mixed Dirichlet-Neumann boundary conditions. When the nonlinear term $f$ is odd and a suitable relation between the perturbation parameter, the limit of $f(\cdot,t)/t$ as $t\to 0$   and the eigenvalues occurs, we establish also a multiplicity result via the pseudo-index theory related to the genus.
\end{abstract}

\section{Introduction}
In this paper we investigate the existence of solutions to the nonlinear problem
\begin{equation}\label{problem}\tag{$P_{\lambda,\mu}$}
\left\{
\begin{tabular}{lcl}
        $(-\Delta)^s u=\lambda u + \mu f(x,u)$ & &in $\Omega$, \\[3pt]
        $\mkern+21.7muB(u)=0$& &on $\partial\Omega$,
\end{tabular}
\right.
\end{equation}
where $\Omega\subset\R^N$ is a bounded domain with a smooth boundary, $N>2s$, $s\in(1/2,1)$, $\lambda,\mu$ are real parameters, and $(-\Delta)^s$ is the spectral fractional Laplacian on $\Omega$ endowed with mixed Dirichlet-Neumann conditions on $\partial\Omega$,
$$
B(u)\vcentcolon=u\chi_{\SD} +\frac{\partial u}{\partial\nu}\chi_{\SN}.
$$
Here $\nu$ is the outward unit normal to $\partial\Omega$, $\chi_A$ denotes the characteristic function of the set $A\subset\partial\Omega$ and moreover the following hypotheses hold,
\begin{itemize}
\item[$(\Omega_1)$] $\SD$ and $\SN$ are smooth $(N-1)$--dimensional submanifolds of $\partial\Omega$;
\item[$(\Omega_2)$] $\SD$ is a closed manifold with positive measure, namely $|\SD|=\alpha\in(0,|\partial\Omega|)$;
\item[$(\Omega_3)$] $\SD\cap\SN=\emptyset$, $\SD\cup\SN=\partial\Omega$ and $\SD\cap\overline{\Sigma}_\mathcal{N}=\Gamma$, where $\Gamma$ is a smooth $(N-2)$--dimensional submanifold of $\partial\Omega$.
\end{itemize}
The nonlinearity $f:\Omega\times\mathbb{R}\mapsto\mathbb{R}$ is assumed to be a Carathéodory function fulfilling
\begin{itemize}
\item[($f_1$)] $\sup\limits_{|t|\leq \xi}|f(\cdot,t)|\in L^\infty(\Omega)$, for all $\xi>0$;
\item[($f_2$)] $\displaystyle \lim_{|t|\to+\infty}\frac{f(x,t)}{t}=0$ uniformly in a.e. $x\in\Omega$;
\item[($f_3)$] $\displaystyle \lim_{t\to 0}\frac{f(x,t)}{t}=\lambda_0\in\R\setminus\{0\}$ uniformly in a.e. $x\in\Omega$.
\end{itemize}

Problem \eqref{problem} can be seen as a nonlinear perturbation of the eigenvalue problem
\[
	\left\{
	\begin{array}{ll}
		(-\Delta)^s u=\lambda u & \text{ in } \Omega, \smallskip\\
	\mkern+21.5mu B(u)=0 & \text{ on } \partial\Omega,
	\end{array}
	\right.
\]
and has been widely studied in the case of Dirichlet boundary conditions ($\Sigma_{\mathcal D}=\partial\Omega$) under different assumptions on $f$, starting from the local case $s=1$ (in this regard, see for instance \cite{rab2009var} for an overview of the subject). The aim of this paper is to show that,  if $\lambda$ does not belong to the spectrum of $(-\Delta)^s$ (non-resonant regime) and $f$ behaves in an asymptotically linear way near the origin, then it is possible to obtain the existence and multiplicity of solutions for \eqref{problem} by means of several variational techniques.

We prove at first that, by assuming $(f_1)-(f_3)$, the energy functional associated with \eqref{problem} has the geometric and the compactness requirements to apply the classical saddle point theorem (see \cite{Rabinowitz}). 
This provides us with the existence of one weak solution for any $\lambda$ different from the eigenvalues of $(-\Delta)^s$ and any $\mu>0$. Adding a symmetry assumption on the nonlinearity and adopting the pseudo-index theory related to the Krasnoselskii genus (cf. \cite{Benci}), together with an abstract result for critical points of even functionals (cf. \cite{Bartolo1983}), we next show that \eqref{problem} has multiple pairs of non-trivial solutions; this time, for $\lambda$ in a range involving $\lambda_0$ and two eigenvalues of $(-\Delta)^s$ (see assumption $(\Lambda)$ below), and any $\mu>0$. Finally, replacing $(f_2)$ by the more general
\begin{itemize}
	\item[$(f_2')$] there exist $q\in\left(1,2N/(N-2s)\right)$, $a_1,a_2>0$ such that
	$$
	|f(x,t)|\leq a_1 +a_2|t|^{q-1}, \quad\text{for all } (x,t)\in\Omega\times\R,
	$$
\end{itemize}
and $(f_3)$ by the more technical
\begin{itemize}	
	\item[$(f_3')$] there exist a non-empty open set $\Omega_1\subseteq\Omega$ and a set $\Omega_2\subset \Omega_1$, with $|\Omega_2|>0$, such that 
	$$
	\limsup_{t\to 0^+}\frac{\essinf_{x\in \Omega_2}f(x,t)}{t}=+\infty \qquad\text{and}\qquad \liminf_{t\to 0^+}\frac{\essinf_{x\in \Omega_1}f(x,t)}{t}=\lambda_0\in\R, 
	$$
\end{itemize}
we are able to prove that the energy functional has a critical point of different nature, i.e., a local minimum. This solution stems from the application of an abstract result of \cite{ric2000a}, and is obtained for $\lambda$ below the first eigenvalue of $(-\Delta)^s$ and $\mu$ in a precisely estimated range (see Theorem \ref{thmsubcritgrowth} and Remark \ref{estimatemulambda}). The aforementioned abstract tool has recently been used to deal with other fractional-Laplacian-type problems (cf. for instance \cite{molservil2023a,molrepvil2017multiple}). Assumption $(f_3')$ can be seen as a lower asymptotically linear growth of $f$ at $0$; as a prototype of such an $f$, one can take
$$
f(x,t):=a(x)|t|^{r-2} t + b(x)t, \quad \text{for all } (x,t)\in\Omega\times\R,
$$
where $r\in(1,2)$, $a,b\in L^\infty(\Omega)$, $\essinf_{\Omega_2} a >\essinf_{\Omega_1} a=0$, and $\essinf_{\Omega_1}b=\lambda_0$. The results of this paper generalize the ones obtained in \cite{barmol2015a} for integro-differential operators of fractional type and pure Dirichlet conditions to the mixed boundary framework. Results of multiplicity of solutions for \eqref{problem}, with $\lambda$ in the nearly resonant regime, have also been recently obtained in \cite{mov2023,colmolortvil2025} by different variational tools.

Let us denote by $\sigma((-\Delta)^s)=\{\Lambda_k\}_{k\in\N}$ the spectrum of $(-\Delta)^s$ with the mixed Dirichlet-Neumann conditions $B(u)=0$ on $\partial\Omega$, each eigenvalue being repeated according its multiplicity,
\begin{equation*}
0<\Lambda_1<\Lambda_2\leq\ldots\Lambda_k\leq\Lambda_{k+1}\leq\ldots 
\end{equation*}
For $\lambda<\Lambda_1$, let us also set
\begin{equation}\label{mMlambda}
	m_\lambda:=\min\left\lbrace \sqrt{\frac{\Lambda_1-\lambda}{\Lambda_1}},1\right\rbrace, \quad M_\lambda:=\max\left\lbrace \sqrt{\frac{\Lambda_1-\lambda}{\Lambda_1}},1\right\rbrace.
\end{equation}
Finally, let us denote by $\kappa_p$, $p\in[1,2N/(N-2s)]$, the best constant for the embedding $H_{\Sigma_{\mathcal{D}}}^s(\Omega) \hookrightarrow L^p(\Omega)$, see Section \ref{functionalsettings} for details.

Our first result, as anticipated, holds for every $\mu>0$. For simplicity, we will assume therefore that $\mu=1$, being the proof easily adapted to general positive $\mu$'s.

\begin{theorem}\label{mainresult}
Assume $(\Omega_1)-(\Omega_3)$. Then,
\begin{itemize}
	\item[$(i)$] under $(f_1)-(f_3)$, for every $\lambda\notin\sigma((-\Delta)^s)$, $\lambda>\Lambda_1$, problem $(P_{\lambda,1})$ has at least one non-trivial weak solution. If, in addition, $f(x,\cdot)$ is odd for a.e. $x\in\Omega$ and
	\smallskip
	\begin{itemize}
		\item[$(\Lambda)$] $\lambda_0 + \lambda<\Lambda_h\leq\Lambda_l<\lambda$, \; for suitable $h,l\in\mathbb{N}$,\, $l\geq h$,
	\end{itemize}
\smallskip
	 then $(P_{\lambda,1})$  has at least $l-h+1$ distinct pairs of non-trivial weak solutions.
	\item[$(ii)$] If we assume $(f_1)-(f_2)$ and replace $(f_3)$ by $(f_3')$, then for every $\lambda<\Lambda_1$, $(P_{\lambda,1})$ has at least one non-trivial weak solution.
\end{itemize}
\end{theorem}

Item $(ii)$ in Theorem \ref{mainresult} is a particular case of the following more general result, in which we provide the existence of one non-trivial solution to \eqref{problem}, as well as the explicit estimate of the threshold of $\mu$ for which such a solution exists.

\begin{theorem}\label{thmsubcritgrowth}
Assume $(\Omega_1)-(\Omega_3)$, $(f_1)$, $(f_2')$ and $(f_3')$. Then, for every $\lambda<\Lambda_1$ and for every $\mu\in(0,\mu_\lambda)$, where
\begin{equation}\label{defmulambda}
\mu_\lambda:=q\sup_{t>0}\frac{t}{q\kappa_1a_1 \frac{\sqrt{2}}{m_\lambda} +\kappa_q^qa_2\left(\frac{\sqrt{2}}{m_\lambda}\right)^q t^{q-1}},
\end{equation}
\eqref{problem} has at least one non-trivial weak solution.
\end{theorem}

We point out that if $q\in(1,2)$ in assumption $(f_2')$, then $\mu_\lambda$ in \eqref{defmulambda} takes the value $+\infty$, in accordance with Theorem \ref{mainresult} -- $(ii)$, which holds for $\mu>0$, see Remark \ref{estimatemulambda}.

\section{Functional framework and pseudo-index theory}\label{functionalsettings}
Let $\{(\lambda_k,\varphi_k)\}$ be the eigenvalues and the eigenfunctions (normalized in the $L^2(\Omega)$-norm), respectively, of $-\Delta$ endowed with homogeneous mixed Dirichlet-Neumann conditions on $\partial\Omega$.  Then $\{(\lambda_k^s,\varphi_k)\}$ are the eigenvalues and the eigenfunctions of $(-\Delta)^s$, respectively (therefore, $\Lambda_k=\lambda_k^s$ for every $k\geq 1$).
As a result, given two smooth functions 
$$
u_i(x)=\sum_{k\geq1}\langle u_i,\varphi_k\rangle_{2}\varphi_k(x), \quad i=1,2,
$$
where $\langle u,v\rangle_2:=\int_{\Omega}uv\,dx$ is the standard scalar product on $L^2(\Omega)$, one has
\begin{equation*}
\langle(-\Delta)^s u_1, u_2\rangle_{2} = \sum_{k\ge 1} \lambda_k^s\langle u_1,\varphi_k\rangle_{2} \langle u_2,\varphi_k\rangle_{2},
\end{equation*}
i.e., the action of the spectral fractional Laplacian on $u_1$ is given by
\begin{equation*}
(-\Delta)^su_1=\sum_{k\ge 1} \lambda_k^s\langle u_1,\varphi_k\rangle_{2}\varphi_k.
\end{equation*}
The operator $(-\Delta)^s$ is then well-defined on the Hilbert space
\begin{equation*}
H_{\Sigma_{\mathcal{D}}}^s(\Omega)\vcentcolon=\left\{u=\sum_{k\ge 1} a_k\varphi_k\in L^2(\Omega):\ u=0\ \text{on }\Sigma_{\mathcal{D}},\ \|u\|_{H_{\Sigma_{\mathcal{D}}}^s}^2\vcentcolon=
\sum_{k\ge 1} a_k^2\lambda_k^s<+\infty\right\}.
\end{equation*}

By \cite[Theorem 11.1]{Lions1972}, if $s\in(0,1/2]$, then $H_0^s(\Omega)=H^s(\Omega)$ (and thus, also
$H_{\Sigma_{\mathcal{D}}}^s(\Omega)=H^s(\Omega)$), while, if $s\in(1/2,1)$, then $H_0^s(\Omega)\subsetneq H^s(\Omega)$. Hence, the range $s\in(1/2,1)$ ensures $H_{\Sigma_{\mathcal{D}}}^s(\Omega)\subsetneq H^s(\Omega)$ and it provides us with the appropriate functional space for problem \eqref{problem}.

The embedding $H_{\Sigma_{\mathcal{D}}}^{s}(\Omega)  \hookrightarrow L^{p}(\Omega)$ is continuous for all $p\in[1,2_s^*]$, where $2^*_s:=2N/(N-2s)$, and compact for $p\in[1,2_s^*)$. We set
$$
\kappa_p:=\sup_{u\in H_{\Sigma_{\mathcal{D}}}^{s}(\Omega)\setminus\{0\}}\frac{\left\|u\right\|_p }{\|u\|_{H_{\Sigma_{\mathcal{D}}}^s}}
$$
(here and in what follows, $\left\|\cdot \right\|_p$, $p\in[1,+\infty]$, stands for the standard $L^p$-norm on $\Omega$). If $p=2_s^*$, the Sobolev constant related to $\Sigma_{\mathcal{D}}$ is defined by
\begin{equation*}
	\widetilde{S}(\Sigma_{\mathcal{D}})=\inf_{u\in
			H_{\Sigma_{\mathcal{D}}}^s(\Omega)\setminus\{0\}}\frac{\|u\|_{H_{\Sigma_{\mathcal{D}}}^s}^2}{\|u\|_{2_s^*}^2}.
\end{equation*}
 Since $\alpha\!\in\!(0,|\partial\Omega|)$ and $H_{0}^s(\Omega)\subsetneq H_{\Sigma_{\mathcal{D}}}^s(\Omega)$, then
$0\!<\widetilde{S}(\Sigma_{\mathcal{D}})\!<S(N,s)$, being $S(N,s)$ the best constant of the embedding $H_0^s(\Omega)\hookrightarrow L^{2^*_s}(\Omega)$.
Indeed, $\widetilde{S}(\Sigma_{\mathcal{D}})\leq 2^{-\frac{2s}{N}}S(N,s)$ (cf. \cite[Proposition 3.6]{colort2019}) and, if $\widetilde{S}(\Sigma_{\mathcal{D}})<2^{-\frac{2s}{N}}S(N,s)$, then $\widetilde{S}(\Sigma_{\mathcal{D}})$ is attained (cf. \cite[Theorem 2.9]{colort2019}).

For every $k\geq 1$, let us denote
$$
\mathbb{H}_k\vcentcolon=\text{span}\{\varphi_1,\ldots,\varphi_k\}
\quad\text{and}\quad
\mathbb{P}_{k}\vcentcolon=\{u\in H_{\SD}^s(\Omega): \left\langle u,\varphi_j\right\rangle_{H_{\SD}^s}=0,\ \forall j=1,\ldots,k\},
$$
where $\left\langle u,v\right\rangle_{H_{\SD}^s}=\left\langle (-\Delta)^{s/2}u,(-\Delta)^{s/2}v\right\rangle_2$ is the scalar product on $H_{\SD}^s(\Omega)$ inducing $\|\cdot\|_{H_{\Sigma_{\mathcal{D}}}^s}$. The following variational characterization of $\Lambda_k$  will be much used in our arguments (cf. \cite[Lemma 6]{mov2023}):
\begin{equation}\label{varcharacteigen}
\Lambda_k=\inf_{u\in \mathbb{P}_{k-1}}\frac{\left\|u\right\|_{H_{\SD}^s}^2}{\left\| u\right\|_{2}^2}=\sup_{u\in \mathbb{H}_{k}}\frac{\left\|u\right\|_{H_{\SD}^s}^2}{\left\| u\right\|_{2}^2}, \quad \text{for any } k\in\mathbb{N}.
\end{equation}
We also recall that, as a consequence of the inclusions $H_{0}^1(\Omega)\subset H_{\Sigma_{\mathcal{D}}}^1(\Omega)\subset H^1(\Omega)$, one has
$\lambda_{k}^N\leq \lambda_{k} \leq \lambda_{k}^{D}\leq\lambda_{k+1}^{N}$,
being $\lambda_{k}^{N}$ and $\lambda_{k}^{D}$ the $k$-th eigenvalue of $-\Delta$ endowed with homogeneous Neumann and Dirichlet boundary conditions, respectively. As a consequence,
$\Lambda_k\to+\infty$ as $k\to+\infty$.

We say that $u\in H_{\SD}^s(\Omega)$ is a weak solution to \eqref{problem} if
\begin{equation*}
\int_\Omega (-\Delta)^{s/2}u (-\Delta)^{s/2} v dx =\lambda\int_\Omega uv dx + \mu\int_\Omega f(x,u)vdx,
\end{equation*}
for all $v\in H_{\SD}^s(\Omega)$. The energy functional $I_{\lambda,\mu}:H_{\Sigma_{\mathcal{D}}}^s(\Omega)\to\R$ associated with \eqref{problem} is given by
\begin{equation*}
I_{\lambda,\mu}(u)\vcentcolon=\frac12\int_{\Omega}|(-\Delta)^{\frac{s}{2}}u|^{2}dx-\frac{\lambda}{2}\int_{\Omega}u^2dx-\mu\int_{\Omega}F(x,u)dx,
\end{equation*}
where $F(x,t):=\int_0^t f(x,\xi)d\xi$ for every $(x,t)\in\Omega\times\R$.

We finally recall the basic ingredients of the pseudo-index theory for even functionals with symmetry group $\mathbb{Z}_2=\{\text{id}, -\text{id}\}$, used in our approach. Let us define
$$
\Sigma\vcentcolon=\{A\subseteq H_{\Sigma_{\mathcal{D}}}^s(\Omega): A\text{ is closed and } -u\in A\text{ if }u\in A\},
$$
and
$$
\mathcal{G}:=\{g\in C^0\left(  H_{\Sigma_{\mathcal{D}}}^s(\Omega),H_{\Sigma_{\mathcal{D}}}^s(\Omega)\right) : g\text{ is odd}\}.
$$
For $A\in\Sigma$, $A\neq \emptyset$, its Krasnoselskii genus is defined as
$$
\gamma(A):=\inf\{n\in\N: \exists\psi\in C^0(A, \mathbb{R}^n\setminus\{0\})\text{ s.t. }\psi(-u)=-\psi(u)\text{ for all } u\in A\}
$$
(we set $\gamma(A)=+\infty$ if such an infimum does not exist and $\gamma(\emptyset)=0$). The index theory  $(\Sigma,\mathcal{G},\gamma)$ related to $\mathbb{Z}_2$ is also called genus. Following \cite{Benci}, we define the pseudo-index related to the genus and $S\in\Sigma$ as the triplet $(S,\mathcal{G}^*,\gamma^*)$, where $\mathcal{G}^*$ is a group of odd homeomorphisms and $\gamma^*:\Sigma\to\mathbb{N}\cup\{+\infty\}$ is the map
$$
\gamma^*(A):= \min\limits_{g\in\mathcal{G}^*}\gamma(g(A)\cap S),\quad\text{for all }A\in \Sigma.
$$
The following critical point theorem plays a key role for deducing multiple solutions to \eqref{problem} ({cf. \cite[Theorem 2.9]{Bartolo1983}).
	
\begin{theorem}\label{thmpseudoindex}
Let $a,b,c_0,c_\infty\in\overline{\R}$, with $-\infty\leq a<c_0<c_\infty<b\leq +\infty$, let $\Phi:H_{\Sigma_{\mathcal{D}}}^s(\Omega)\to\R$ be an even functional, $(\Sigma,\mathcal{G},\gamma)$ the genus theory on $H_{\Sigma_{\mathcal{D}}}^s(\Omega)$, $S\in\Sigma$, $(S,\mathcal{G}^*,\gamma^*)$ the pseudo-index related to the genus and $S$, with
\[\mathcal{G}^*\vcentcolon=\{g\in\mathcal{G}: g \text{ is a bounded homeomorphism and } g(u)=u\text{ if }u\notin\Phi^{-1}((a,b))\}.\]
Assume that
\begin{itemize}
\item[$(i)$]  $\Phi$ satisfies the Palais-Smale condition $(PS)_c$ for all $c\in(a,b)$, i.e., any sequence $\{u_j\}\subset H_{\Sigma_{\mathcal{D}}}^s(\Omega)$ satisfying $\Phi(u_j)\to c$ and $\Phi'(u_j)\to 0$ as $j\to +\infty$, has a convergent subsequence;
\item[$(ii)$] $S\subseteq \Phi^{-1}([c_0,+\infty))$;
\item[$(iii)$] there exists $\tilde{k}\in\mathbb{N}$ and $\tilde{A}\in\Sigma$ such that $\tilde{A}\subseteq \{u\in H_{\Sigma_{\mathcal{D}}}^s(\Omega):\Phi(u)\leq c_\infty\}$ and $\gamma^*(\tilde{A})\geq \tilde{k}$.
\end{itemize}
Then, for all $i\in\{1,2,\ldots,\tilde{k}\}$, setting $\Sigma_i^*:=\{A\in\Sigma:\gamma^*(A)\geq i\}$, the numbers
\begin{equation}\label{charactci}
c_i=\inf\limits_{A\in \Sigma_i^*}\sup\limits_{u\in A}\Phi(u)
\end{equation}
are critical values for $\Phi$ and one has $c_0\leq c_1\leq\ldots\leq c_{\tilde{k}}\leq c_\infty$. Furthermore, if $c=c_i=\ldots=c_{i+r}$ with $i\geq 1$ and $i+r\leq \tilde{k}$, then 
$\gamma\left( \left\lbrace u\in H_{\Sigma_{\mathcal{D}}}^s(\Omega): \Phi(u)=c,\ \Phi'(u)=0\right\rbrace \right) \geq r+1$.
\end{theorem}

Problems with asymptotically linear data have also been investigated under similar assumptions in other works; see, for instance, the papers \cite{ABMB,BMB,BDM} and references therein. Reference \cite{MBRS}, on the other hand, should be regarded as a primary source for fractional problems.


\section{Proof of Theorems \ref{mainresult} and \ref{thmsubcritgrowth}} 
\begin{proof}[Proof of Theorem \ref{mainresult}]
Let us first address the existence part in $(i)$. By $(f_1)$ and $(f_2)$ there exists $C_1>0$ such that
\[
|F(x,t)|\leq C_1(1+t^2),\quad\text{ for all } t\in\R \text{ and for a.e.}\ x\in\Omega,
\]
and by \eqref{varcharacteigen}, taking $k\in\N$ large enough,
\begin{equation*}
I_{\lambda,1}(u)\geq\frac{1}{2}\left(1-\frac{\lambda+C_2}{\Lambda_{k+1}}\right)\left\|u\right\|_{H_{\SD}^s}^2-C_3\geq C_4,
\end{equation*}
for all $u\in\mathbb{P}_{k}$ and for suitable $C_2,C_3>0$ and $C_4\in\R$. In addition, by using again $(f_1)-(f_2)$, fixed $\varepsilon>0$ there exists $C_\varepsilon>0$ such that
$$
I_{\lambda,1}(u)\leq\frac{1}{2}\left\|u\right\|_{H_{\SD}^s}^2-\frac{\lambda}{2}\left\|u\right\|_2^2+\frac{\varepsilon}{2}\left\|u\right\|_2^2 +C_\varepsilon \left\|u\right\|_2,\quad\text{for all}\ u\in H_{\SD}^s(\Omega).
$$
Taking $\Lambda_k<\lambda$ 
and $\varepsilon>0$ such that $\Lambda_k+\varepsilon<\lambda$, we then obtain
\begin{equation}\label{post}
I_{\lambda,1}(u)\leq\frac{1}{2}\left(\Lambda_k+\varepsilon-\lambda\right)\left\|u\right\|_2^2+C_\varepsilon \left\|u\right\|_2,\quad\text{for all}\ u\in\mathbb{H}_{k},
\end{equation}
and  being $\mathbb{H}_{k}$ finite-dimensional, 
there exist $r,C_5>0$ such that
\begin{equation*}
I_{\lambda,1}(u)\leq-C_5, \quad\text{ for all } u\in\mathbb{H}_{k} \text{ with } \left\| u\right\|_{\Hs}=r.
\end{equation*}
Then, by the saddle point theorem (cf. \cite[Theorem 4.6]{Rabinowitz}), $(P_{\lambda,1})$ has at least a weak solution.

Next, let us focus on the multiplicity of solutions in $(i)$. By $(f_2)-(f_3)$ we obtain
$$
\lim\limits_{|t|\to+\infty}\frac{F(x,t)}{t^2}=0\quad\text{ and }\quad \lim\limits_{t\to 0}\frac{F(x,t)}{t^2}=\frac{\lambda_0}{2},
$$
uniformly in a.e. $x\in\Omega$. Let us notice that condition $(\Lambda)$ forces $\lambda_0<0$. Thus, for every $\varepsilon>0$ there exists $r_\varepsilon\geq 1$ and $\delta_\varepsilon>0$ such that, for a.e. $x\in\Omega$, one has
\begin{equation}\label{F_1}
|F(x,t)|\leq \frac{\varepsilon}{2}t^2,\quad\text{if }|t|>r_\varepsilon, \qquad\text{and}\qquad
\left| F(x,t)-\frac{\lambda_0}{2}t^2\right| \leq \frac{\varepsilon}{2}t^2,\quad\text{if }|t|<\delta_\varepsilon.
\end{equation}
Moreover, by $(f_1)$, fixed $m\in[0,2_s^*-2)$, there exists $k_{\varepsilon}>0$ such that
\begin{equation}\label{F_3}
|F(x,t)|\leq k_{\varepsilon} |t|^{m+2},\quad\text{ for a.e. } x\in\Omega \text{ and for }\delta_\varepsilon\leq |t|\leq r_\varepsilon.
\end{equation}

Gathering together \eqref{F_1}-\eqref{F_3} it follows that
$$
F(x,t)\leq\frac{\lambda_0+\varepsilon}{2}t^2 + k_\varepsilon|t|^{m+2} \quad \text{for a.e. } x\in\Omega \text{ and for all } t\in\R.
$$
Thus, choosing $\varepsilon\in\left( \max\{0,-\lambda-\lambda_0\},\Lambda_h-\lambda-\lambda_0\right)$, we obtain
\begin{align*}
I_{\lambda,1}(u)&\geq\frac{1}{2}\left\| u\right\|_{H_{\SD}^s}^2-\frac{\lambda + \lambda_0+\varepsilon}{2}\left\| u\right\|_2^2-C_\varepsilon\left\| u\right\|_{H_{\SD}^s}^{m+2}\\
&\geq\frac{1}{2}\left(1-\frac{\lambda+\lambda_0+\varepsilon}{\Lambda_{h}}\right)\left\| u\right\|_{H_{\SD}^s}^2-C_\varepsilon\left\| u\right\|_{H_{\SD}^s}^{m+2},
\end{align*}
for all $u\in \mathbb{P}_{h-1}$ and for some $C_\varepsilon>0$,
and, as a result, it is possible to determine a sufficiently small $\varrho>0$ and $c_0>0$ such that
 \begin{equation}\label{energyonrhosphere}
I_{\lambda,1}(u)\geq 
c_0, \quad \text{for all } u\in \mathbb{P}_{h-1} \text{ with } \left\| u\right\|_{\Hs}=\varrho.
\end{equation}
If now 
we choose  $\varepsilon>0$ such that $\Lambda_l+\varepsilon<\lambda$, by \eqref{post}
there exists $c_\infty=c_\infty(\varepsilon)>c_0$ such that
\begin{equation}\label{relaz2}
I_{\lambda,1}(u)\leq c_\infty,\quad\text{for all } u\in \mathbb{H}_{l}.
\end{equation}

Following the same arguments as \cite{barmol2015a}, it is easily seen that, under the assumptions $(f_1)-(f_2)$ and that $\lambda\notin\sigma((-\Delta)^s)$, $I_{\lambda,1}$ satisfies $(PS)_c$ at any level $c\in\mathbb\R$ . Also, the evenness of $F$ implies the one of $I_{\lambda,1}$. Then, if $\varrho,c_0,c_\infty$ are as in \eqref{energyonrhosphere} and \eqref{relaz2}, setting $\partial B_\varrho:=\left\lbrace u\in H_{\Sigma_{\mathcal{D}}}^{s}(\Omega): \left\| u\right\|_{\Hs}=\varrho \right\rbrace$, and considering the pseudo-index theory $(\partial B_\varrho\cap\mathbb{P}_{h-1},\mathcal{G}^*,\gamma^*)$ related to the genus, $\partial B_\varrho\cap\mathbb{P}_{h-1}$ and $I_{\lambda,1}$, 
we obtain the following estimate,
$$
\gamma(\mathbb{H}_l\cap g(\partial B_\varrho\cap\mathbb{P}_{h-1})) \geq \operatorname{dim} \mathbb{H}_l -\operatorname{codim} \mathbb{P}_{h-1}, \quad \text{for every } g\in\mathcal{G}^*,
$$
that is,
$$
\gamma^*(\mathbb{H}_l) \geq l-h+1.
$$
Here we used the fact that, due to \cite[Theorem A.2]{Bartolo1983}, given the genus theory $(\Sigma,\mathcal G,\gamma)$ on $H_{\Sigma_{\mathcal{D}}}^{s}(\Omega)$
and two closed subspaces $H_1,H_2\subset H_{\Sigma_{\mathcal{D}}}^{s}(\Omega)$ with $\max\{\operatorname{dim}H_1,\operatorname{codim}H_2\}<+\infty$, then
$$
\gamma(H_1\cap g(\partial B\cap H_2))\geq \operatorname{dim}H_1 -\operatorname{codim}H_2
$$
for every bounded $g\in\mathcal G$ and every open bounded symmetric neighbourhood $B$ of 0 in $H_{\Sigma_{\mathcal{D}}}^{s}(\Omega)$. Thus, by Theorem \ref{thmpseudoindex}, where in particular $\tilde A=\mathbb{H}_l$ and $S=\partial B_\varrho\cap\mathbb{P}_{h-1}$, $I_{\lambda,1}$ has at least $l-h+1$ distinct pairs of critical points corresponding to at most $l-h+1$ distinct critical values $c_i$, the latter characterized by \eqref{charactci}. This concludes the proof of $(i)$.

The proof of $(ii)$ is a consequence of the one of Theorem \ref{thmsubcritgrowth}, since $(f_2)$ implies $(f_2')$.
\end{proof}

\medskip

\begin{proof}[Proof of Theorem \ref{thmsubcritgrowth}]
If $\lambda<\Lambda_1$, the functional
$$
u\mapsto \left\| u\right\|_{H_{\SD},\lambda}:=\left(\left\|u \right\|^2_{H_{\SD}}-\lambda  \left\| u\right\|_2^2\right)^\frac{1}{2}
$$
defines a norm on $\Hs$ equivalent to $\left\| \cdot\right\|_{H_{\SD}}$, as one has
\begin{equation}\label{normeequiv}
m_\lambda \|u\|_{H_{\SD}}\leq \|u\|_{H_{\SD},\lambda}\leq M_\lambda  \|u\|_{H_{\SD}},
\end{equation}
with $m_\lambda$ and $M_\lambda$ defined in \eqref{mMlambda}
For $u\in\Hs$, we write  $I_{\lambda,\mu}(u)=\Phi_\lambda(u) - \mu\Psi(u)$, where
$$
\Phi_\lambda(u):=\frac{1}{2}\left\| u\right\|_{H_{\SD},\lambda}^2 \quad\text{and}\quad \Psi(u):=\int_\Omega F(x,u) dx.
$$
Having in mind to use Theorem 2.1 in \cite{ric2000a}, we notice at once that $\Psi$ is sequentially weakly continuous on $\Hs$. Indeed, given $\{u_j\}\subset\Hs$ such that $u_j\rightharpoonup u\in\Hs$, by Sobolev embeddings $u_j\to u$ in $L^q(\Omega)$ 
and $|u_j(x)| \leq w(x)$
for all $j\in\N$, for a.e. $x\in\Omega$ and for some $w\in L^q(\Omega)$. By $(f_2')$ we also deduce that
\[
|F(x,u_j)| \leq a_1|u_j(x)| +\frac{a_2}{q}|u_j(x)|^q \leq a_1 w(x) + \frac{a_2}{q}(w(x))^q
\]
for all $j\in\N$ and for a.e. $x\in\Omega$. Then, by dominated convergence,
$$
\lim_{j\to +\infty}\int_\Omega F(x,u_j)dx =\int_\Omega F(x,u) dx,
$$
as claimed.

Now, take $\mu\in(0,\mu_\lambda)$; there will exist $\bar{t}>0$ such that
\begin{equation}\label{maggmulambda}
\mu < \frac{q\bar{t}}{q\kappa_1a_1 \frac{\sqrt{2}}{m_\lambda} +\kappa_q^qa_2\left(\frac{\sqrt{2}}{m_\lambda}\right)^q \bar{t}^{q-1}}.
\end{equation}
If $r>0$, $u\in\Hs$ and $\Phi_\lambda(u)<r$, by \eqref{normeequiv} we deduce that
$$
\left\|u\right\|_{H_{\SD}} <\frac{\sqrt{2r}}{m_\lambda}.
$$
By $(f_2')$ we then obtain
$$
\Psi(u)  \leq a_1\left\| u\right\|_1 +\frac{a_2}{q}\left\| u\right\|_q^q
 \leq \kappa_1a_1 \frac{\sqrt{2r}}{m_\lambda} + \kappa_q^q  \frac{a_2}{q} \left( \frac{\sqrt{2r}}{m_\lambda}\right)^q,
$$
and therefore
$$
\sup_{u\in\Phi_\lambda^{-1}(-\infty,r)}\Psi(u) \leq \kappa_1a_1 \frac{\sqrt{2}}{m_\lambda}r^\frac{1}{2} + \kappa_q^q  \frac{a_2}{q} \left( \frac{\sqrt{2}}{m_\lambda}\right)^q r^\frac{q}{2}.
$$
If $\eta_{\lambda}:(0,+\infty)\to [0,+\infty)$ is the function defined by
\begin{equation*}
	\eta_{\lambda}(r):=\inf_{u\in\Phi_{\lambda}^{-1}\left((-\infty,r)\right)}
	\frac{\sup_{v\in\Phi_{\lambda}^{-1}\left( (-\infty,r)\right) }\Psi(v)-\Psi(u)}{r-\Phi_{\lambda}(u)}, \quad \text{for all } r>0,
\end{equation*}
we then get
\begin{align*}
	\eta_\lambda(\bar{t}^2) & = \inf_{u\in\Phi_{\lambda}^{-1}\left((-\infty,\bar{t}^2)\right)}
	\frac{\sup_{v\in\Phi_{\lambda}^{-1}\left( (-\infty,\bar{t}^2)\right) }\Psi(v)-\Psi(u)}{\bar{t}^2-\Phi_{\lambda}(u)}\\
	& \leq \kappa_1a_1 \frac{\sqrt{2}}{m_\lambda}\bar{t}^{-1} + \kappa_q^q  \frac{a_2}{q} \left( \frac{\sqrt{2}}{m_\lambda}\right)^q \bar{t}^{q-2},
\end{align*}
where we used the fact that $0\in \Phi_{\lambda}^{-1}\left((-\infty,\bar{t}^2)\right)$ and $\Phi_\lambda(0)=\Psi(0)=0$. By \eqref{maggmulambda}
we derive at once that
$$
\eta_\lambda(\bar{t}^2)< \frac{1}{\mu},
$$
and therefore, by \cite[Theorem 2.1]{ric2000a}, 
the restriction of $I_{\lambda,\mu}$ to $\Phi_{\lambda}^{-1}\left( (-\infty, \bar{t}^2)\right)$ has a global minimum 
which is a critical point of $I_{\lambda,\mu}$ in the whole $\Hs$ 
This shows that \eqref{problem} has a weak solution $u_\mu$ for any $\lambda<\Lambda_1$ and any $\mu\in (0,\mu_\lambda)$.


To conclude the proof, let us show that for any $\mu\in (0, \mu_\lambda)$, $u_\mu\not\equiv 0$ in $\Omega$. Fix $\bar\mu\in (0, \mu_\lambda)$,
and let $\bar{t}>0$ such that \eqref{maggmulambda} is satisfied with $\mu=\bar{\mu}$. By what seen before, $I_{\lambda,\bar{u}}$ has a local minimum $u_{\bar{\mu}}\in\Phi_\lambda^{-1}(-\infty,\bar{t}^2)$.
By $(f_3')-(i)$, one has
\begin{equation*}
	\lim_{j\rightarrow +\infty}\frac{\essinf_{x\in \Omega_2}F(x, \xi_j)}{\xi_j^2}=+\infty
\end{equation*}
along a suitable sequence $\{\xi_j\}\subset(0,+\infty)$ such that $\xi_j\to 0^+$ as $j\to
+\infty$, i.e.,
\begin{equation*}
\essinf_{x\in \Omega_2}F(x, \xi_j)>M\xi_j^2,
\end{equation*}
for any $M>0$ and sufficiently large $j\in\N$. Now, let $\Omega_3\subset \Omega_2$, $|\Omega_3|>0$, let $v\in\Hs$ verify $0\leq v\leq 1$, $v=1$ in $\Omega_3$, $v=0$ in $\Omega\setminus \Omega_1$, and set $w_j:=\xi_j v$ for any $j\in \N$. By $(f_3')-(ii)$, for every $\varepsilon>0$ there exists $\varrho_\varepsilon>0$ such that
\begin{equation*}
	\essinf_{x\in \Omega_1}\,F(x,t) \geq(\lambda_0-\varepsilon)t^2, \quad \text{for all } 0<t<\rho_\varepsilon.
\end{equation*}
Therefore, for large $j\in\N$, we obtain
\begin{equation}\label{conto1}
\begin{split}
\frac{\Psi(w_j)}{\Phi_{\lambda}(w_j)}  &= \frac{\displaystyle \int_{\Omega_3}F(x, w_j) dx +\int_{\Omega_1\setminus
\Omega_3} F(x, w_j)dx}{\Phi_{\lambda}(w_j)}
 \geq \frac{\displaystyle 2M|\Omega_3| \xi_j^2+ 2 \int_{\Omega_1\setminus \Omega_3} F(x,\xi_j v)dx}{\xi_j^2\|v\|_{H_{\SD},\lambda}^2}\\
& \geq \frac{2M|\Omega_3| \xi_j^2 + 2(\lambda_0-\varepsilon)\xi_j^2\displaystyle\int_{\Omega_1\setminus \Omega_3} v^2 dx}{\xi_j^2\|v\|_{H_{\SD},\lambda}^2} =\frac{2M|\Omega_3| + 2(\lambda_0-\varepsilon)\displaystyle \int_{\Omega_1\setminus \Omega_3} v^2dx}{\|v\|_{H_{\SD},\lambda}^2}.
\end{split}
\end{equation}

Choosing
$$
M>\displaystyle\max\left\lbrace 0, -2\lambda_0\int_{\Omega_1\setminus \Omega_3}v^2dx\right\rbrace |\Omega_3|^{-1}
\quad \text{and} \quad
0<\varepsilon<M|\Omega_3|\left(2\int_{\Omega_1\setminus \Omega_3}v^2dx\right)^{-1} +\lambda_0,
$$
from \eqref{conto1} we obtain
\[
\frac{\Psi(w_j)}{\Phi_{\lambda}(w_j )}
 \geq \frac{2 M|\Omega_3| + 2\lambda_0 \displaystyle\int_{\Omega_1\setminus \Omega_3} v^2 dx - M|\Omega_3| -2\lambda_0 \int_{\Omega_1\setminus \Omega_3}v^2 dx}{\|v\|_{H_{\SD},\lambda}^2}
 = \frac{M|\Omega_3|}{\|v\|_{H_{\SD},\lambda}^2},
\]
still for $j$ large enough. The arbitrariness of $M$ implies that
\begin{equation}\label{wj}
\limsup_{j\to +\infty}\frac{\Psi(w_j)}{\Phi_{\lambda}(w_j)}=+\infty.
\end{equation}
As $\|w_j\|_{H_{\SD},\lambda}\to 0$ for $j\to +\infty$, we get of course that $w_j\in \Phi_\lambda^{-1}(-\infty,\bar{t}^2)$ for large $j\in\N$. 
Moreover, by \eqref{wj},
\begin{equation*}
I_{\lambda,\bar\mu}(w_j)=\Phi_{\lambda}(w_j)-\bar\mu\Psi(w_j)<0,
\end{equation*}
for sufficiently large $j\in\N$ and thus
\begin{equation*}
I_{\lambda,\bar{\mu}}(u_{\bar{\mu}})\leq I_{\lambda,\bar{\mu}}(w_j)<0 =I_{\lambda,\bar{\mu}}(0),
\end{equation*}
i.e., $u_{\bar\mu}\not\equiv 0$ in $\Hs$. The arbitrariness of $\bar \mu$ shows that  $u_\mu\not \equiv 0$ for any $\mu\in (0, \mu_\lambda)$ and this concludes the proof.
\end{proof}

\begin{remark}\label{estimatemulambda}
{\rm By direct computations, it turns out that, according to the range of $q$, $\mu_\lambda$ in \eqref{defmulambda} attains the following values:
$$
\mu_\lambda=
\left\lbrace
\begin{array}{ll}
	+\infty & \text{ if } q\in(1,2),\smallskip\\
	\displaystyle\frac{m_\lambda^2}{\kappa_2^2 a_2} & \text{ if } q=2,\smallskip\\
	\displaystyle\frac{m_\lambda^2}{2(q-1)} \left[ \frac{q}{\kappa_q^q a_2}\left( \frac{q-2}{\kappa_1 a_1}\right)^{q-2} \right]^\frac{1}{q-1} & \text{ if } q\in(2,2^*_s).
\end{array}	
\right.
$$
}	
\end{remark}

\begin{remark}
{\rm If $f(\cdot,0)\neq 0$ in $\Omega$, then the trivial function $u\equiv0$ does not solve \eqref{problem}, so the local minimum $u_\mu$ is certainly non-trivial and assumption $(f_3')$ in Theorem \ref{thmsubcritgrowth} can be dropped.}
\end{remark}

\begin{remark}
{\rm It can be proved that the last conclusion of Theorem \ref{mainresult} holds true also if $\lambda_0=+\infty$ in assumption $(f_3')-(ii)$. In this case, as a model for $f$, we can take $f(x,t)=a(x)|t|^{r-2}t +b(x)t$ for all $(x,t)\in\Omega\times\R$, with $r\in(1,2)$, $a,b\in L^\infty(\Omega)$ and $\essinf_\Omega a>0$.}
\end{remark}

\section*{Acknowledgement}
\noindent G. Molica Bisci is funded by the European Union - NextGenerationEU within the framework of PNRR Mission 4 - Component 2 - Investment 1.1 under the Italian Ministry of University and Research (MUR) program PRIN 2022 - grant number 2022BCFHN2 - Advanced
theoretical aspects in PDEs and their applications - CUP: H53D23001960006.\\
A. Ortega is partially funded by Vicerrectorado de Investigación, Transferencia del Conocimiento y Divulgación Científica of Universidad Nacional de Educación a Distancia under research project Talento Joven UNED 2025, Ref: 2025/00151/001.\\
L. Vilasi is a member of the Gruppo Nazionale per l'Analisi Matematica, la Probabilità e le loro Applicazioni (GNAMPA) of the Istituto Nazionale di Alta Matematica (INdAM). This work is partially funded by the ``INdAM - GNAMPA Project CUP E5324001950001".

\section*{Conflict of interest}
\noindent On behalf of all authors, the corresponding author states that there is no conflict of interest. 


\end{document}